\documentclass[reqno]{amsart}
\usepackage{mathrsfs,amssymb,mathrsfs, multirow,setspace}
\usepackage{amsmath}
\usepackage{listings}
\usepackage{pdfpages}
\usepackage{caption}
\usepackage[a4paper, total={7in, 9in}]{geometry}
\usepackage{enumerate}
\usepackage{subfig}
\usepackage{tikz}
\theoremstyle{plain}
\usepackage{graphicx}
\usepackage{mathtools}
\usepackage[utf8]{inputenc}

\newtheorem{theorem}{Theorem}[section]
\newtheorem{lemma}[theorem]{Lemma}
\newtheorem{proposition}[theorem]{Proposition}
\newtheorem{corollary}[theorem]{Corollary}

\theoremstyle{definition}

\theoremstyle{remark}
\newtheorem{remark}[theorem]{Remark}

\input xy
\xyoption{all}

\begin{document}
	\title[Strong Resolving Graphs of Clean Graphs of Commutative Rings]{Strong Resolving Graphs of Clean Graphs of Commutative Rings}

   \author[Praveen Mathil, Jitender Kumar]{Praveen Mathil, Jitender Kumar$^{^{*}}$}
   \address{Department of Mathematics, Birla Institute of Technology and Science Pilani, Pilani-333031, India}
 \email{maithilpraveen@gmail.com,  jitenderarora09@gmail.com}

\begin{abstract}
Let $R$ be a ring with unity. The clean graph $\textnormal{Cl}(R)$ of a ring $R$ is the simple undirected graph whose vertices are of the form $(e,u)$, where $e$ is an idempotent element and $u$ is a unit of the ring $R$ and two vertices $(e,u)$, $(f,v)$ of $\textnormal{Cl}(R)$ are adjacent if and only if $ef = fe =0$ or $uv = vu=1$. In this manuscript, for a commutative ring $R$, first we obtain the strong resolving graph of $\textnormal{Cl}(R)$ and its independence number. Using them, we determine the strong metric dimension of the clean graph of an arbitrary commutative ring. As an application, we compute the strong metric dimension of $\textnormal{Cl}(R)$, where $R$ is a commutative Artinian ring.

\end{abstract}
\keywords{Clean elements, strong metric dimension, Artinian ring, reduced ring  \\ *  Corresponding author}
\maketitle

\section{Introduction}

The exploration of algebraic structures through the analysis of graphs associated with them has evolved into a compelling and substantive research area over the past few decades. This comprehensive research area not only signifies the intricate interplay between algebraic structures and graph theory but also underscores the escalating significance of this interdisciplinary pursuit within the broader mathematical landscape. Numerous researchers have studied the graph associated with ring structures (cf. \cite{afkhami2012generalized, akbari2009total, anderson2008total, anderson1999zero, ashrafi2010unit, beck1988coloring, behboodi2011annihilating,  maimani2008comaximal}). Units and idempotent elements of a ring play a vital role in the ring theory, attracting authors to study the algebraic graphs defined using these elements such as idempotent graph \cite{akbari2013idempotent}, unit graph \cite{ashrafi2010unit}, idempotent-divisor graph \cite{kimball2018idempotent},  graph with respect to idempotents of a ring \cite{razaghi2021graph}, etc.


The notion of the metric dimension of a graph is defined by Harary and Melter \cite{harary1976metric}. It enables robotics engineers to create a moving robot that can determine its present location in a network of navigating agents. Other than robot navigation, the metric dimension of a graph holds considerable relevance across diverse domains, including but not limited to image processing, combinatorial optimization, chemistry, network security and so on (cf. \cite{khuller1996landmarks,oellermann2007strong,slater1975leaves}). The strong metric dimension, which originates from the metric dimension, was introduced by Seb\H{o} and Tannier \cite{sebHo2004metric} in 2004. They illustrated some applications of the strong metric dimension of a graph in combinatorial searching. Numerous researchers investigated the strong metric dimension for various classes of graphs, see for instance, \cite{kuziak2013strong, kuziak2015strong, rodriguez2014strong}. To gain further insight into the strong metric dimension of a graph, one can refer to \cite{kratica2014strong}. Due to the wide range of applications and noteworthy history of the strong metric dimension of a graph, algebraic graph theorists determine the strong metric dimension of graphs associated with algebraic structures (see \cite{ebrahimi2021strong,ma2018strong,ma2021strong,nikandish2021metric,nikandish2022strong,zhai2023metric}).




An element of a ring is said to be clean if it can be written as the sum of an idempotent element and a unit. A ring is called a \emph{clean ring} if all the elements of the ring are clean. The concept of clean rings was introduced by Nicholson \cite{nicholson1977lifting} in 1977. Many researchers conducted thorough investigations on them (cf. \cite{anderson2002commutative, han2001extensions, nicholson2004clean, nicholson2004rings}). We denote a clean element as $(e,u)$, where $e$ is an idempotent and $u$ is a unit. The clean graph of a ring is defined by Habibi \emph{et al.} \cite{habibi2021clean}. The clean graph $\textnormal{Cl}(R)$ of the ring $R$ is the simple undirected graph whose vertex set is the set of all the clean elements of the ring $R$ and two vertices $(e,u)$, $(f,v)$ are adjacent if and only if $ef = fe =0$ or $uv = vu=1$. The subgraph of $\textnormal{Cl}(R)$ induced by the set $\{ (0,u) \in V(\textnormal{Cl}(R))\}$ is denoted by $\textnormal{Cl}_1(R)$. Note that $\textnormal{Cl}_1(R)$ is a complete graph. The graph $\textnormal{Cl}_2(R)$ is the subgraph of $\textnormal{Cl}(R)$ induced by the set $\{ (e,u) \in V(\textnormal{Cl}(R)) ~ | ~ e \neq 0 \}$. In \cite{habibi2021clean}, the authors studied the relation between the graph-theoretic properties and ring-theoretic properties of the clean graph of a ring. Further, they investigated the clique number, the chromatic number, the domination number and the independence number of the clean graph. The embedding of clean graphs of commutative rings on various surfaces is investigated in \cite{ramanathanaclassification}. In this paper, we study the strong resolving graph and investigate the strong metric dimension of the clean graph of a ring. The paper is arranged as follows. Section \ref{section2} comprises basic definitions and necessary results used throughout the paper. In Section \ref{section3}, we study the strong resolving graph of the clean graph and determine its independence number. Section \ref{section4} comprises the investigation of the strong metric dimension of the clean graph of an arbitrary commutative ring. Moreover, we determine the strong metric dimension of the clean graph of a commutative Artinian ring.

\section{Preliminaries}\label{section2}

A \emph{graph} $\Gamma$ is an ordered pair $(V(\Gamma), E(\Gamma))$, where $V(\Gamma)$ is the set of vertices and $E(\Gamma)$ is the set of edges of $\Gamma$. Two distinct vertices $u, v \in V(\Gamma)$ are $\mathit{adjacent}$ in $\Gamma$, denoted by $u \sim v$, if there is an edge between $u$ and $v$. Otherwise, we write it as $u \nsim v$. Let $\Gamma$ be a graph. By a finite graph $\Gamma$, we mean the vertex set of $\Gamma$ is finite. A graph $\Gamma' = (V(\Gamma'), E(\Gamma'))$ is said to be a \emph{subgraph} of $\Gamma$ if $V(\Gamma') \subseteq V(\Gamma)$ and $E(\Gamma') \subseteq E(\Gamma)$. If $X \subseteq V(\Gamma)$, then the subgraph $\Gamma(X)$ induced by the set $X$ is the graph with vertex set $X$ and two vertices of $\Gamma(X)$ are adjacent if and only if they are adjacent in $\Gamma$.  Let  $\Gamma_1$ and  $\Gamma_2$ be two graphs. The \emph{union} $\Gamma_1 \cup \Gamma_2$ is the graph with $V(\Gamma_1 \cup \Gamma_2) = V(\Gamma_1) \cup V(\Gamma_2)$ and $E(\Gamma_1 \cup \Gamma_2) = E(\Gamma_1) \cup E(\Gamma_2)$. If $V(\Gamma_1) \cap V(\Gamma_2) = \emptyset$, then $\Gamma_1 \cup \Gamma_2 = \Gamma_1 + \Gamma_2$. A \emph{path} in a graph is a sequence of distinct vertices with the property that each vertex in the sequence is adjacent to the next vertex of it. The distance $d(x,y)$ between any two vertices $x$ and $y$ of $\Gamma$ is the number of edges in a shortest path between $x$ and $y$. The \emph{diameter} \rm{diam}$(\Gamma)$ of a connected graph $\Gamma$ is the maximum of the distances between any two vertices in $\Gamma$. A graph $\Gamma$ is said to be \emph{complete} if any two vertices are adjacent in $\Gamma$. The \emph{neighbourhood} $N(x)$ of the vertex $x$ is the set of all the adjacent vertices with $x$. \emph{Vertex cover} of a graph $\Gamma$ is a subset $S$ of vertex set of $\Gamma$ such that every edge of $\Gamma$ has one end in $S$. The smallest cardinality of a vertex cover is called the \emph{vertex covering number} of the graph $\Gamma$ and it is denoted by $\alpha(\Gamma)$. A subset of a vertex set of the graph $\Gamma$ is said to be an \emph{independent set} if no two vertices are adjacent in it. The \emph{independence number} $\beta(\Gamma)$ of the graph $\Gamma$ is the cardinality of a largest independent set of $\Gamma$. In a graph $\Gamma$, a vertex $u$ is \emph{maximally distant} from a vertex $v$ if $d(v,w) \le d(u,v)$ for every $w \in N(u)$. Vertices $u$ and $v$ are said to be \emph{mutually maximally distant} if both $u$ and $v$ are maximally distant from each other. The boundary of a graph $\Gamma$ is defined as

\[ \partial(\Gamma) = \{ u \in V(\Gamma) ~ | ~ \text{there exists $ v \in V(\Gamma)$ such that $u$ and $v$ are mutually maximally distant}  \}. \]

The strong resolving graph $\Gamma_{SR}$ of the graph $\Gamma$ is the graph with the vertex set $\partial(\Gamma)$ and two vertices $u$ and $v$ are adjacent if and only if $u$ and $v$ are mutually maximally distant.

 A vertex $z$ resolves two distinct vertices $x$ and $y$ of $\Gamma$ if $d(x, z) \neq d(y, z)$. A subset $W$ of $V(\Gamma)$ is said to be a \emph{resolving set} of $\Gamma$ if every pair of distinct vertices of $\Gamma$ is resolved by some vertex in $W$. A vertex $w$ in $\Gamma$ strongly resolves two vertices $u$ and $v$ if there exists a shortest path between $u$ and $w$ containing $v$, or there exists a shortest path between $v$ and $w$ containing $u$. A subset $S$ of $V(\Gamma)$ is a \emph{strong resolving set} of $\Gamma$ if every two distinct vertices of $\Gamma$ are strongly resolved by some vertex of $S$. The \emph{strong metric dimension} $\text{sdim}(\Gamma)$ of $\Gamma$  is the smallest cardinality of a strong resolving set in $\Gamma$. The following theorems are important for later use.

 \begin{theorem}{\cite[Theorem 2.1]{oellermann2007strong}}\label{vertexcover}
     For a connected graph $\Gamma, \textnormal{sdim}(\Gamma) = \alpha(\Gamma_{SR})$.
 \end{theorem}

 \begin{theorem}\label{Gallaithoerem}
     For any graph $\Gamma$ of order $n$, $\alpha(\Gamma) + \beta(\Gamma) = n$.
 \end{theorem}

Throughout the paper, the ring $R$ is a commutative ring with unity. For basic definitions of ring theory, we refer the reader to  \cite{atiyah1969introduction}. A ring $R$ is said to be \emph{local} if it has a unique maximal ideal $\mathcal{M}$. For the ring $R$, $U(R)$ denotes the set of all unit elements of the ring $R$. Moreover, $U'(R) = \{ u \in U(R) ~ | ~ u^2 =1 \}$ and $U''(R) = U(R) \setminus U'(R)$. An element $e$ of the ring $R$ is said to be an idempotent element if $e^2 = e$. By $\textnormal{Id}(R)$, we mean the set of all idempotent elements of the ring $R$. Two idempotent elements $e$ and $f$ are said to be orthogonal if $e f =0$. Set of orthogonal elements of $R$ which has maximum cardinality is denoted by $\textnormal{Id}_{\perp}(R)$. Also, $\textnormal{Id}(R)^* = \textnormal{Id}(R) \setminus \{0,1 \}$ and $\textnormal{Id}_{\perp}(R)^* = \textnormal{Id}_{\perp}(R) \setminus \{0 \}$. We use $\textnormal{Max}(R)$ for the set of all the maximal ideals of the ring $R$. The following propositions are easy to prove.

\begin{proposition}\label{twononselfinvertible}
    Let $(R, +, \cdot)$ be a commutative ring with unity. Then $|U''(R)| =2$ if and only if either $(U(R), \cdot) \cong (\mathbb{Z}_3, +)$ or $(U(R), \cdot) \cong (\mathbb{Z}_4, +)$.
\end{proposition}

\begin{proposition}\label{selfinvertibleunits}
Let $(R, +, \cdot)$ be any finite commutative ring with unity. Then $U''(R) = \emptyset$ if and only if either $\left(U(R), \cdot \right) \cong \mathbb{Z}_1$ or $\left(U(R), \cdot \right) \cong \left( \mathbb{Z}_2 \times \mathbb{Z}_2 \times \cdots \times \mathbb{Z}_2, + \right)$. 
 \end{proposition}



Local rings up to the order $8$ are listed in the following table.  
 \begin{center}
\begin{tabular}{|c| c|}
\hline
 $ |\textbf{R}|$ & $\textbf{R}$ \\
 \hline
$2$   & $\mathbb{Z}_2$ \\
\hline
 $ 3$  & $\mathbb{Z}_3$ \\
\hline
$4$  & $\mathbb{F}_4$, $\mathbb{Z}_4$, $\dfrac{\mathbb{Z}_2[x]}{\langle x^2 \rangle}$ \\
\hline
$5$  & $\mathbb{Z}_5$ \\
\hline
$ 7$  & $\mathbb{Z}_7$ \\
\hline
$ 8$  & $\mathbb{F}_8$, $\mathbb{Z}_8$, $\dfrac{\mathbb{Z}_2[x]}{\langle x^3 \rangle}$, $\dfrac{\mathbb{Z}_2[x,y]}{\langle x^2, xy, y^2 \rangle}$, $\dfrac{\mathbb{Z}_4[x]}{\langle 2x, x^2 \rangle}$, $\dfrac{\mathbb{Z}_4[x]}{\langle 2x, x^2-2 \rangle}$ \\
\hline
\end{tabular}
\captionof{table}{Local rings up to the order $8$}
\label{localrings}
\end{center}

\section{Strong Resolving Graph of Clean Graph}\label{section3}

In this section, we obtain the strong resolving graphs of the clean graphs $\textnormal{Cl}(R)$ and $\textnormal{Cl}_2(R)$, respectively. Note that if the ring $R$ has no non-trivial idempotents, then the graph $\textnormal{Cl}_2(R)$ is disconnected. Consequently, we investigate the strong resolving graph of the graph $\textnormal{Cl}_2(R)$, where $R$ is a ring containing non-trivial idempotent elements. To study the strong resolving graph, first we identify the mutually maximally distant vertices of the graphs $\textnormal{Cl}(R)$ and $\textnormal{Cl}_2(R)$, respectively. We begin with the following theorem. 



\begin{theorem}\label{mmdclr}
    Let $R$ be a ring with no non-trivial idempotents. Then the following holds for the graph $\textnormal{Cl}(R)$.

    \begin{enumerate}[\rm(i)]
        \item If $|U(R)|=1$, then $\textnormal{Cl}(R) \cong K_2$.
        \item If $|U(R)| \ge 2$, then two vertices $(e,u)$ and $(f, v)$ are mutually maximally distant if and only if either $e=f=0$ or $e=f=1$.
    \end{enumerate}
\end{theorem}

\begin{proof}
   Let $R$ be a ring with no non-trivial idempotents. If $|U(R)| =1$, then $V(\textnormal{Cl}(R)) = \{ (0,1), (1,1)\}$ and so $\textnormal{Cl}(R) \cong K_2$. Now suppose that $|U(R)| \ge 2$. Let $(e,u)$ and $(f, v)$ be two distinct vertices of $\textnormal{Cl}(R)$. Assume that $(e,u)$ and $(f, v)$ are mutually maximally distant vertices. Let $e=0$. It follows that $d((e,u), x) =1$ for all $x \in V(\textnormal{Cl}(R))$ and so $d((e,u), (f,v)) =1$. Since $(e,u)$ and $(f, v)$ are mutually maximally distant, $d((f,v), y) \le d((e,u), (f,v))$ for every $y \in N((e,u))$. It implies that $d((f,v), y) =1$ for all $y \in N((e,u))$. Now suppose that $f=1$. Then choose $y = (1,w) \in N((e,u))$ such that if $v=1$, then $w \in U(R) \setminus \{ 1\}$, and if $v \neq 1$, then $w =1$. Consequently, $d((f,v), y)  \neq 1$, a contradiction. It implies that $f=0$.


   Next, let $e=1$. Now suppose that $f=0$. Then $d((e,u), (f,v)) =1$. Choose $(1,r) \in V(\textnormal{Cl}(R))$ such that if $u=1$, then $r \in U(R) \setminus \{ 1\}$ and if $u \neq 1$, then $r =1$. Then note that $(1,r) \in N((f,v))$. Also, $d((e,u), (1,r)) =2$, which is a contradiction to the fact that $(e,u)$ and $(f, v)$ are mutually maximally distant. Consequently, $f=1$.
   
   Conversely, suppose that either $e=f=0$ or $e=f=1$. We show that $(e,u)$, $(f, v)$ are mutually maximally distant vertices. First assume that $e=f=0$. Then $d((e,u), (f,v)) =1$. Moreover, $d((e,u), x) =1$ and $d((f,v), x) = 1$ for all $x \in V(\textnormal{Cl}(R))$. Consequently, $(e,u)$, $(f, v)$ are mutually maximally distant. Next, let $e=f=1$. Suppose that $u v \neq 1$. Then $(e,u) \nsim (f, v)$. Since $\textnormal{diam(\textnormal{Cl}(R))} = 2$, we obtain $d((e,u), (f,v)) =2$. Also, $d((e,u), x) \le 2$ and $d((f,v), x) \le 2$ for all $x \in V(\textnormal{Cl}(R))$. Thus, $(e,u)$, $(f, v)$ are mutually maximally distant. Now, let $e=f=1$ and $u v =1$. It follows that $d((e,u), (f,v)) =1$. Assume that $(g,w) \in N((e,u)) \setminus \{ (f,v) \}$. Then $g = 0$ and so $d((g,w), (f,v)) =1$. Similarly, $d((e,u), z) =1$ for all $ z \in N((f,v))$. Thus, $(e,u)$, $(f, v)$ are mutually maximally distant.
\end{proof}

\begin{theorem}\label{mmdcl2r}
    Let $R$ be a ring with some non-trivial idempotents. Then the following holds for the graph $\textnormal{Cl}_2(R)$.
    \begin{enumerate}[\rm(i)]
        \item If $|U(R)|=1$, then $\textnormal{Cl}_2(R)$ is a complete graph.
        \item If $|U(R)| \ge 2$, then two vertices $(e,u)$ and $(f, v)$ are mutually maximally distant if and only if one of the following holds
         \begin{enumerate}[\rm(a)]
         \item $e=f=1$  and $u v \neq 1$.
         \item $e,f \neq 1$, $e f \neq 0$ and $u v \neq 1$.
         \item $e =1$, $f \neq 1$ and $u =v \in U''(R)$.
         \end{enumerate}
    \end{enumerate}
\end{theorem}

\begin{proof}
    Let $R$ be a commutative ring with some non-trivial idempotents. If $|U(R)| =1$, then $V(\textnormal{Cl}_2(R)) = \{ (e,1) ~ | ~ e \in \textnormal{Id}(R) \setminus \{0 \} \}$ and so $\textnormal{Cl}_2(R) \cong K_{|\textnormal{Id}(R)|-1}$. Now suppose that $|U(R)| \ge 2$. Let $(e,u)$, $(f, v)$ be two distinct vertices of $\textnormal{Cl}_2(R)$. Assume that $(e,u)$, $(f, v)$ are mutually maximally distant. Now, we prove the result in the following cases.

    \noindent\textbf{Case 1.} $e=f=1$. In this case, we show that $u v \neq 1$. On the contrary, suppose that $u v = 1$. Then $d((e,u), (f,v)) = 1$. Consider the vertex $(f',u)$ such that $ f' \in \textnormal{Id}(R) \setminus \{0,1 \} $. Then note that $(f',u) \in N((f,v))$. Also, $d((e,u), (f',u)) \ge 2$. It follows that $(e,u)$ and $(f,v)$ are not mutually maximally distant, a contradiction. Consequently, $u v \neq 1$.

    \noindent\textbf{Case 2.} $e,f \neq 1$. First, suppose that $e f = 0$. It follows that $d((e,u), (f,v)) = 1$. Consider $(f,w) \in V(\textnormal{Cl}_2(R))$ such that $w v \neq 1$. Then note that $(f,w) \in N((e,u))$. Also, $d((f,v), (f,w)) \ge 2$. It implies that $(e,u)$ and $(f,v)$ are not mutually maximally distant, a contradiction. Thus, $e f \neq 0$.
    
    Next, suppose that $u v=1$. It follows that $d((e,u), (f,v)) = 1$. If $e =f$, then consider the vertex $(1, v) \in V(\textnormal{Cl}_2(R))$. Then note that $(1,v) \in N((e,u))$ and $d((f,v), (1,v)) \ge 2$. It follows that $(e,u)$ and $(f,v)$ are not mutually maximally distant, a contradiction. Consequently, $u v \neq 1$. Now let $e \neq f$. If $(1-e)f \neq 0$, then consider $(1-e, w) \in N((e,u))$ such that $vw \neq 1$. Then note that $d((f,v), (1-e,w)) \ge 2$, a contradiction. If $(1-e)f = 0$, then consider $(1-f, w') \in N((f,v))$ such that $uw' \neq 1$. Observe that $(1-f)e  \neq 0$ and so $d((e,u), (1-f,w')) \ge 2$, which is a contradiction to the fact that $(e,u)$ and $(f,v)$ are mutually maximally distant. It concludes that $u v \neq 1$.

    \noindent\textbf{Case 3.} $e =1$, $f \neq 1$. Suppose that $uv =1$. Then $d((e,u), (f,v)) = 1$. Consider the vertex $(1-f, w) \in N((f,v))$ such that $uw \neq 1$. Note that, $d((e,u), (1-f, w)) \ge 2$. It follows that $(e,u)$ and $(f,v)$ are not mutually maximally distant, a contradiction. Consequently, $u v \neq 1$.
    
    Now suppose that $u \neq v$. Then $(e,u) \sim (1-f, u^{-1}) \sim (f,v)$ and so $d((e,u), (f,v)) = 2$. Consider the vertex $(1, v^{-1}) \in N((f,v))$. Notice that $d((e,u), (1, v^{-1}) ) \ge 2$. Assume that $d((e,u), (1, v^{-1}) ) = 2$. It follows that there exist $(g, w') \in V(\textnormal{Cl}_2(R))$ such that $(e,u) \sim (g, w') \sim (1, v^{-1})$. Since $g 1 \neq 1$, we have $u w' = v^{-1} w' =1$, which is not possible. Therefore, $d((e,u), (1, v^{-1}) ) = 3$, which is a contradiction to the fact that the vertices $(e,u)$ and $(f,v)$ are mutually maximally distant. Therefore, $u = v $. Also, $u v \neq 1$ implies that $u,v \notin U'(R)$. Thus, $u = v \in U''(R) $.

   Conversely, first let $(e,u)$, $(f, v) \in V(\textnormal{Cl}_2(R))$ such that $e=f=1$ and $u v \neq 1$.  It implies that $d((e,u), (f,v)) \ge 2$.  Assume that $d((e,u), (f,v)) = 2$. Then there exists $(s, \delta) \in V(\textnormal{Cl}_2(R))$ such that $(e,u) \sim (s, \delta) \sim (f, v)$. Since $1s  \neq 0$, we have $u \delta = 1$ and $v \delta = 1$, which is not possible. Hence,  $d((e,u), (f,v)) = \textnormal{diam}(\textnormal{Cl}_2(R)) = 3$. Moreover, $(e, u) \sim (h, u^{-1}) \sim (1-h, v^{-1}) \sim (f,v)$, where $h$ is a non-trivial idempotent of $R$. Thus, the vertices $(e,u)$ and $(f, v)$ are mutually maximally distant.  

   Now let $e,f \neq 1$, $e f \neq 0$ and $u v \neq 1$. Clearly, $d((e,u), (f,v)) \ge 2$. First, we show that if $e \neq 1$, then $d((e,u), x) \le 2$ for all $x \in V(\textnormal{Cl}_2(R))$. Let $x = (h,w) \in V(\textnormal{Cl}_2(R))$. If $(e,u) \nsim (h,w)$, then $(e,u) \sim (1-e, w^{-1}) \sim (h,w)$. Therefore, $d((e,u), x) \le 2$ for all $x \in V(\textnormal{Cl}_2(R))$. Since $f \neq 1$, we have   $d((f,v), y) \le 2$ for all $y \in V(\textnormal{Cl}_2(R))$. Thus, $d((e,u), (f,v)) = 2$ and so $(e,u)$ and $(f, v)$ are mutually maximally distant vertices. 
   
   Finally, let $e =1$, $f \neq 1$ and $u =v \in U''(R)$. Then note that $d((e,u), (f,v)) \ge 2$. Similarly, if $f \neq 1$, notice that $d((e,u), (f,v)) = 2$. Also, $d((f,v), x) \le 2$ for all $x \in V(\textnormal{Cl}_2(R))$. Assume that $(h', w') \in N((f,v))$. If $h' =1$, then $w' v =1$ and so $d((e,u), (h', w')) = 1$. Now if $h' \neq 1$, then one can observe that $d((e,u), (h', w')) \le 2$. Thus, the vertices $(e,u)$ and $(f, v)$ are mutually maximally distant.

   This completes our proof.
\end{proof}

Note that $\textnormal{Cl}(R) = \textnormal{Cl}_1(R) \vee \textnormal{Cl}_2(R)$, where $\textnormal{Cl}_1(R)$ is a complete subgraph of $\textnormal{Cl}(R)$. Thus, if $d(x,y) \in \{ 2,3 \}$ in $\textnormal{Cl}_2(R)$, then  $d(x,y) =2$ in $\textnormal{Cl}(R)$.

\begin{theorem}\label{mmdclrwithid}
    Let $R$ be a ring with some non-trivial idempotents. Then the following holds for the graph $\textnormal{Cl}(R)$.

    \begin{enumerate}[\rm(i)]
        \item If $|U(R)|=1$, then $\textnormal{Cl}(R)$ is a complete graph.
        \item If $|U(R)| \ge 2$, then two vertices $(e,u)$ and $(f, v)$ are mutually maximally distant if and only if either $e=f=0$ or $e,f \neq 0$, $e f \neq 0$ and $u v \neq 1$.
    \end{enumerate}
\end{theorem}

\begin{proof}

    Let $R$ be a ring with some non-trivial idempotents. If $|U(R)| =1$, then $V(\textnormal{Cl}(R)) = \{ (e,1) ~ | ~ e \in \textnormal{Id}(R)\}$ and so $\textnormal{Cl}(R) \cong K_{|\textnormal{Id}(R)|}$. Now suppose that $|U(R)| \ge 2$. Note that $\textnormal{Cl}(R) = \textnormal{Cl}_1(R) \vee \textnormal{Cl}_2(R)$, where $\textnormal{Cl}_1(R)$ is a complete subgraph of $\textnormal{Cl}(R)$. Thus, if $d(x,y) \in \{ 2,3 \}$ in $\textnormal{Cl}_2(R)$, then  $d(x,y) =2$ in $\textnormal{Cl}(R)$. Let $x=(e,u)$ and $y=(f, v)$ be two distinct vertices of $\textnormal{Cl}(R)$. If $x \nsim y$, then $d(x,y) = \textnormal{diam}(\textnormal{Cl}(R)) = 2$. Consequently, $x$ and $y$ are mutually maximally distant vertices. 
    
    Now let $x \sim y$. Suppose that $e, f \neq 0$. Then by Theorem \ref{mmdcl2r}, note that $x$ and $y$ are not mutually maximally distant vertices in the graph $\textnormal{Cl}_2(R)$. It follows that $x$ and $y$ are not mutually maximally distant vertices in $\textnormal{Cl}(R)$. Now suppose that $e=0$ and $f \neq 0$. Since $|U(R)| \ge 2$, we can choose $v' \in U(R) \setminus \{v\}$ such that $v v' \neq 1$. Then note that $(1, v') \in N(x)$ and $d(y, (1,v')) =2 > d(x,y)$. Therefore, $x$ and $y$ are not mutually maximally distant vertices. Finally, let $e=f=0$. Then for each $z \in V(\textnormal{Cl}(R))$, we have $d(x,z) = d(y,z) =d(x,y) =1$. Consequently, $x$ and $y$ are mutually maximally distant vertices. Thus, $(e,u)$ and $(f,v)$ are mutually maximally distant vertices if and only if either $e=f=0$ or $(e,u) \nsim (f,v)$. This completes our proof. 
\end{proof}

Now we study the structure of the strong resolving graph of the clean graphs. Note that when $|U(R)| =1$, then $\textnormal{Cl}(R)$ and $\textnormal{Cl}(R)$ are complete graphs. Consequently, the graph $\textnormal{Cl}(R)_{SR}$ and $\textnormal{Cl}_2(R)_{SR}$ are also complete graphs and isomorphic to the graphs $\textnormal{Cl}(R)$ and $\textnormal{Cl}(R)$, respectively. Now we study the strong resolving graph of the clean graphs when $|U(R)| \ge 2$. For this, define a graph $K$ such that $V(K) = \{(e, u) ~ | ~ e \in \textnormal{Id}(R) \setminus \{ 0,1\}, ~ u \in U(R) \}$ and two vertices $(e, u)$, $(f, v)$ are adjacent if and only if $e f \neq 0$, $u v \neq 1$. Next, consider a graph $H'$ such that $V(H') = V(\textnormal{Cl}_2(R))$ and two vertices $(e, u)$, $(f, v) \in V(H')$ are adjacent if and only if one of the following holds:
  \begin{enumerate}[\rm(i)]
         \item $e=f=1$  and $u v \neq 1$.
         \item $e,f \neq 1$, $e f \neq 0$ and $u v \neq 1$.
         \item $e =1$, $f \neq 1$ and $u =v \in U''(R)$.
         \end{enumerate}

Using the above-defined graphs $K$ and $H'$, we have the following structure theorems of the strong resolving graph of the graphs $\textnormal{Cl}_{2}(R)$ and $\textnormal{Cl}(R)$.

\begin{theorem}\label{graphgcl2r}
   Let $R$ be a ring with some non-trivial idempotents. Then the following holds for the graph $\textnormal{Cl}_2(R)$.
     \begin{enumerate}[\rm(i)]
        \item If $U''(R) = \emptyset$ and $\textnormal{Cl}_2(R)$ is not a complete graph, then $\textnormal{Cl}_2(R)_{SR} = K_{|U(R)|} + K$.
        \item If $U''(R) \neq \emptyset$ and $\textnormal{Cl}_2(R)$ is not a complete graph, then $\textnormal{Cl}_2(R)_{SR} \cong H'$.
         \item $V(\textnormal{Cl}_2(R)) = V(\textnormal{Cl}_2(R)_{SR})$.
    \end{enumerate}
   
\end{theorem}

   

\begin{proof}
       \rm(i) Let $U''(R) = \emptyset$. Since $\textnormal{Cl}_2(R)$ is not a complete graph, we have $|U(R)| \ge 2$. Define a graph $H$ such that $V(H) = \{(1, u) ~ | ~ u \in U(R) \}$ and two vertices $(1, u)$, $(1, v)$ are adjacent if and only if $u v \neq 1$. First, we show that $\textnormal{Cl}_2(R)_{SR} = H + K$. To show that we prove that $V(\textnormal{Cl}_2(R)_{SR}) = V(H) \cup V(K)$ and $x \sim y$ in $\textnormal{Cl}_2(R)_{SR}$ if and only if $x \sim y$ in $H$ or $x \sim y$ in $K$. It is clear that $V(\textnormal{Cl}_2(R)_{SR}) \subseteq V(\textnormal{Cl}_2(R)) = V(H) \cup V(K)$.
        Let $(1,u') \in H$. Since, $|U(R)| \ge 2$, we can choose $v' \in U(R)$ such that $u' v' \neq 1$. By Theorem \ref{mmdcl2r}, it follows that $(1,v')$ and $(1,u')$ are mutually maximally distant vertices and so $(1,u') \in V(\textnormal{Cl}_2(R)_{SR}$. It implies that $V(H) \subseteq V(\textnormal{Cl}_2(R)_{SR}$. Now let $(e,u') \in K$. Then for each $v' \in U(R) \setminus \{ u' \}$, where $u' v' \neq 1$, note that $(e,u')$ and $(e,v')$ are mutually maximally distant vertices. Thus, $(e,u') \in V(\textnormal{Cl}_2(R)_{SR})$ and so $V(K) \subseteq V(\textnormal{Cl}_2(R)_{SR}$. Consequently, $V(H) \cup V(K) \subseteq V(\textnormal{Cl}_2(R)_{SR}$ and so $V(\textnormal{Cl}_2(R)_{SR}) = V(H) \cup V(K)$. By Theorem \ref{mmdcl2r}, note that if $x \sim y$ in $H$, then $x \sim y$ in $\textnormal{Cl}_2(R)_{SR}$. Also, if $x \sim y$ in $K$, then $x \sim y$ in $\textnormal{Cl}_2(R)_{SR}$.

       Now, let $x = (e,u)$ and $y=(f,v)$ be two adjacent vertices in $ \textnormal{Cl}_2(R)_{SR}$. Then, $x$ and $y$ are mutually maximally distant. By Theorem \ref{mmdcl2r}, it implies that $uv \neq 1$ and $e f \neq 0$. Note that, since $|U(R)| \ge 2$, we can always choose $u,v \in U(R)$ such that $uv \neq 1$. If $e = 1$, then $x \in V(H)$. Also, by Theorem \ref{mmdcl2r}, we get $f=1$ and so $y \in V(H)$. Since $u v \neq 1$, we conclude that $x \sim y$ in $H$. Now, if $e \neq 1$, then $x \in V(K)$. Also, by Theorem \ref{mmdcl2r}, we obtain $f \neq 1$ and so $y \in V(K)$. Since $u v \neq 1$, we have that $x \sim y$ in $K$. Thus, $\textnormal{Cl}_2(R)_{SR} = H \cup K$. Also, $V(H) \cap V(K) = \emptyset$ follows that $\textnormal{Cl}_2(R)_{SR} = H + K$. Since $U''(R) = \emptyset$, for two distinct vertices $(1,u)$ and $(1,v)$ of $H$, we have $uv \neq 1$. It follows that $H$ is a complete graph $K_{|U(R)|}$. This completes our proof.


       
       \rm(ii) Let $U''(R) \neq \emptyset$. It is clear that $V(\textnormal{Cl}_2(R)_{SR}) \subseteq V(\textnormal{Cl}_2(R)) = V(H')$. Now let $(e,u) \in V(H')$. Since, $|U(R)| \ge 2$, we can choose $v (\neq u) \in U(R)$ such that $u v \neq 1$. By Theorem \ref{mmdcl2r}, the vertices $(e, u)$ and $(e,v)$ are mutually maximally distant. Also, $(e,v) \in V(H)$. Thus $(e,u) \in V(\textnormal{Cl}_2(R)_{SR})$. Consequently, $V(\textnormal{Cl}_2(R)_{SR}) = V(H')$. By Theorem \ref{mmdcl2r} and the adjacency of the graph $H'$, we have $\textnormal{Cl}_2(R)_{SR} \cong H'$.
       
      \rm(iii) By the first and second parts of the proof, it is easy to show that $V(\textnormal{Cl}_2(R)) = V(\textnormal{Cl}_2(R)_{SR})$.
\end{proof}


\begin{theorem}\label{clrsrnoidempotent}
   Let $R$ be a commutative ring with no non-trivial idempotent such that $\textnormal{Cl}(R)$ is not a complete graph. Then $\textnormal{Cl}(R)_{SR} = 2K_{|U(R)|}$. Moreover, $V(\textnormal{Cl}(R)) = V(\textnormal{Cl}(R)_{SR})$.

\end{theorem} 

\begin{proof}
Define the sets $S$ and $S'$ such that $S = \{(0, u) ~ | ~ u \in U(R)\}$ and $S' = \{(1, v) ~ | ~ v \in U(R)\}$. Then note that $V(\textnormal{Cl}(R)_{SR}) \subseteq V(\textnormal{Cl}(R)) = S \cup S'$. Now we show that $ S \cup S' \subseteq V(\textnormal{Cl}(R)_{SR})$. Let $(0, u) \in S$. Since $\textnormal{Cl}(R)$ is not a complete graph, we have $|U(R)| \ge 2$. and so let $u' \in U(R) \setminus \{ u\}$. Then $(0,u)$ and $(0,u')$ are mutually maximally distant vertices of $\textnormal{Cl}(R)$. It implies that $(0,u) \in \textnormal{Cl}(R)_{SR}$ and so $S \subseteq V(\textnormal{Cl}(R)_{SR})$. Now let $(1,v) \in S'$. Consider $v' \in U(R)\setminus \{ v\}$. Then $(1,v)$ and $(1,v')$ are mutually maximally distant vertices of $\textnormal{Cl}(R)$. It follows that $(1,v) \in \textnormal{Cl}(R)_{SR}$ so $S' \subseteq V(\textnormal{Cl}(R)_{SR})$. Thus, $ S \cup S' \subseteq V(\textnormal{Cl}(R)_{SR})$ and hence $V(\textnormal{Cl}(R)_{SR}) = S \cup S'$. Also, Theorem \ref{mmdclr} and the fact $S \cap S' = \emptyset$ follows that $\textnormal{Cl}(R)_{SR} = \textnormal{Cl}(S)_{SR} + \textnormal{Cl}(S')_{SR}$. By Theorem \ref{mmdclr}, observe that $\textnormal{Cl}(S)_{SR} \cong K_{|S|}  = K_{|U(R)|}$. Also, $\textnormal{Cl}(S')_{SR} \cong K_{|S'|}  = K_{|U(R)|}$. Thus, $\textnormal{Cl}(R)_{SR} = 2K_{|U(R)|}$.
\end{proof}

Define the graph $G$ such that $V(G) = \{(e, u) ~ | ~ 0 \neq e \in \textnormal{Id}(R), ~ u \in U(R) \}$ and two vertices $(e, u)$, $(f, v)$ are adjacent if and only if $e f \neq 0$, $u v \neq 1$.

\begin{theorem}\label{graphgclr}
   Let $R$ be a commutative ring with some non-trivial idempotents such that $\textnormal{Cl}(R)$ is not a complete graph. Then $\textnormal{Cl}(R)_{SR} = G + K_{|U(R)|}$, where $G$ is a connected graph. Moreover, $V(\textnormal{Cl}(R)) = V(\textnormal{Cl}(R)_{SR})$.   
\end{theorem}

\begin{proof}
     Define a set $S = \{(0,u) ~ | u \in U(R)\}$. Then note that $V(\textnormal{Cl}(R)_{SR}) \subseteq V(\textnormal{Cl}(R)) = S \cup V(G)$.
     Let $(0,u') \in S$. Then take $v' \in U(R) \setminus \{ u'\}$. By Theorem \ref{mmdclrwithid} note that $(0,u')$ and $(0,v')$ are mutually maximally distant vertices in $\textnormal{Cl}(R)$. Also, $(0,v') \in S$. It follows that $(0, u') \in V(\textnormal{Cl}(R)_{SR})$ and so $S \subseteq V(\textnormal{Cl}(R)_{SR})$. Now let $(f,w) \in V(G)$. Since $|U(R)| \ge 2$, there exists $w' \in U(R)$ such that $w w' \neq 1$. Then $(f,w)$ and $(f,w')$ are mutually maximally distant vertices of $\textnormal{Cl}(R)$. Thus, $(f, w) \in V(\textnormal{Cl}(R)_{SR})$ and so $V(G) \subseteq V(\textnormal{Cl}(R)_{SR})$. Consequently,  $S \cup V(G) \subseteq V(\textnormal{Cl}(R)_{SR})$ and so $V(\textnormal{Cl}(R)_{SR})= S \cup V(G)$.

     Next, we show that $\textnormal{Cl}(R)_{SR}= \textnormal{Cl}(S)_{SR} + G $. Let $x = (e,u)$ and $y=(f,v)$ be two adjacent vertices in $\textnormal{Cl}(R)_{SR}$. By Theorem \ref{mmdclrwithid}, note that if $x \in S$, then $y \in S$. Also, $\textnormal{Cl}(S)_{SR}$ is a complete subgraph of $\textnormal{Cl}(R)_{SR}$. Now, if $x \notin S$ i.e, $e \neq 0$, then $x \in V(G)$. Again by Theorem \ref{mmdclrwithid}, we have $y \in V(G)$. Since, $x \sim y $ in $\textnormal{Cl}(R)_{SR}$, we have $x$ and $y$ are mutually maximally distant vertices in $\textnormal{Cl}(R)$. It follows that $e f \neq 0$, $u v \neq 1$ and so $x \sim y$ in the graph $G$. Moreover, if $x' \sim y'$ in $G$, then note that $x' \sim y'$ in $\textnormal{Cl}(R)_{SR}$. Thus, $\textnormal{Cl}(R)_{SR}= \textnormal{Cl}(S)_{SR} \cup G = K_{|U(R)|} \cup G= K_{|S|} \cup G$. Note that $S \cap V(G) = \emptyset$. Hence, $\textnormal{Cl}(R)_{SR}=  K_{|U(R)|} + G$.

    Now, we show that $G$ is a connected graph. let $(g, u')$ and $(h, v')$ be two distinct vertices of the graph $G$. If $g h \neq 0$ and $u' v' \neq 1$, then there is nothing to prove. If $g h =0$ but $u' v' \neq 1$, then  $(g, u') \sim (1,v') \sim (1,u') \sim (h, v')$. Now let $u' v' = 1$. Since $|U(R)| \ge 2$, we can choose $w' \in U(R) \setminus \{u',v'\}$ such that $w' u' \neq 1$ and $w' v' \neq 1$. If $g h \neq 0$ and $u' v' = 1$ or $g h =0$ and $u' v' = 1$, then $(g, u') \sim (1,w') \sim (h, v')$. Thus, $G$ is a connected graph.
 \end{proof}

Now we calculate the independence number of $\textnormal{Cl}_2(R)_{SR}$, when $|U(R)| \ge 2$.
 \begin{theorem}\label{independentsetcl2r}
    Let $R$ be a commutative ring with some non-trivial idempotents. Then the following holds
    \begin{enumerate}[\rm(i)]
        \item If $U''(R) = \emptyset$, then $\beta(\textnormal{Cl}_2(R)_{SR}) = |\textnormal{Id}(R)^*|+1$.
        \item If $|U''(R)| =2$ and $\max \{2|\textnormal{Id}_{\perp}(R)^*|, |\textnormal{Id}(R)^*| \} = |\textnormal{Id}(R)^*|$, then $\beta(\textnormal{Cl}_2(R)_{SR}) = |\textnormal{Id}(R)^*| +2$.
        \item If $|U''(R)| =2$ and $\max \{2|\textnormal{Id}_{\perp}(R)^*|, |\textnormal{Id}(R)^*| \} = 2|\textnormal{Id}_{\perp}(R)^*|$, then $\beta(\textnormal{Cl}_2(R)_{SR}) = 2|\textnormal{Id}_{\perp}(R)^*| +1$.
        \item If $|U''(R)| > 2$, then $\beta(\textnormal{Cl}_2(R)_{SR}) = \max \{2|\textnormal{Id}_{\perp}(R)^*|, |\textnormal{Id}(R)^*| \} +2$.
    \end{enumerate}
\end{theorem}

\begin{proof}
We prove the result in the following cases.

\noindent\textbf{Case 1.} $U''(R) = \emptyset$. By Theorem \ref{graphgcl2r}, we have $\beta(\textnormal{Cl}_2(R)_{SR}) = \beta (K) + \beta(K_{|U(R)|}) = \beta(K) + 1$. For $u \in U(R)$, define a set $A_u = \{ (e,u) ~ | ~ e \in \textnormal{Id}(R)^* \}$. Notice that $A_u$ is an independent set of the graph $K$. Now, we show that $A_u$ is an independent set of maximum size of the graph $K$. Let $B$ be any other arbitrary independent set of the graph $K$. Then we prove that $|B| \le |A_u| = |\textnormal{Id}(R)^*|$. For each $u \in U(R)$, consider the set 
    \[ B_u = \{ (e,u) \in B ~ | ~ \text{for some} ~ e \in  \textnormal{Id}(R)^* \}. \]
Note that $B_u$, where $u \in U(R)$, forms a partition of the set $B$. Also, for distinct $v$ and $v' \in U(R)$, $(f,v)$ and $(f, v')$ are adjacent in $\textnormal{Cl}_2(R)_{SR}$, where $f \in \textnormal{Id}(R)^*$. It follows that both $(f,v)$ and $(f, v')$ do not belong to $B$. Consequently, $\sum_{u \in U(R)} |B_u| \le |\textnormal{Id}(R)^*|$. It follows that $|B| \le |A_u|$. It implies that $A_u$ is an independent set of maximum size of the graph $K$, and so $\beta (K) = |A_u| = |\textnormal{Id}(R)^*|$. Thus, $\beta(\textnormal{Cl}_2(R)_{SR}) = |\textnormal{Id}(R)^*|+1$.

\noindent\textbf{Case 2.} $U''(R) \neq \emptyset$. By Theorem \ref{graphgcl2r}, we have $\beta(\textnormal{Cl}_2(R)_{SR}) = \beta(H')$. Let $X$ be an independent set of maximum size of the graph $H'$. Let $X_{(1)} = \{(e, u) \in X ~ | ~ e =1 \} $. Let $(1,u)$ be an element of the set $X_{(1)}$. If $u \in U'(R)$, then for all $(1,v) \in V(\textnormal{Cl}_2(R)_{SR}) \setminus (1,u)$, we have $(1,u) \sim (1,v)$ in $\textnormal{Cl}_2(R)_{SR}$. It follows that $|X_{(1)}| \le 1$. Now let $u \in U''(R)$. Since $(1,u) \nsim (1, u^{-1})$ in $\textnormal{Cl}_2(R)_{SR}$, note that $(1, u^{-1})$ can be an element of the set $X_{(1)}$. Now, if $(1,w)$ be another element of the set $X_{(1)}$, then $w =u^{-1}$. It implies that $|X_{(1)}| \le 2$. Let $X' = X \setminus X_{(1)}$. Now we calculate $|X'|$ and $|X_{(1)}|$. First, let $Y = \{(e,u) \in V(\textnormal{Cl}_2(R)_{SR}) ~ | ~ e \neq 1 \}$. Next, we find $\beta(\textnormal{Cl}_2(Y)_{SR})$. For each $u \in U(R)$, we define the set 
\[  
I_u =
 \begin{cases}
\{ (e,u) ~ | ~  e \in \textnormal{Id}(R)^* \}& \;\; \;\;  \text{if} ~ u \in U'(R)\\
 
\{ (e,u), (e, u^{-1}) ~ | ~ 0 \neq e \in \textnormal{Id}_{\perp}(R)^*\} & \;\;  \;\; \text{if}  ~ u \in U''(R).
 \end{cases}
    \]

    \textbf{Claim:} $I_u$ is an independent set of maximum size of the graph $\beta(\textnormal{Cl}_2(Y)_{SR})$.
    
    It is clear that for each $u \in U(R)$, $I_u$ is an independent set of $\textnormal{Cl}_2(Y)_{SR}$. Also, note that if $ u \in U'(R)$, then $|I_u| =|\textnormal{Id}(R)^*|$. Otherwise, $|I_u| =2|\textnormal{Id}_{\perp}(R)^* |$.  
Suppose that $m = \max\{|I_u| ~ | ~ u \in U(R) \}$. Then
    \[
    m =
 \begin{cases}
  |\textnormal{Id}(R)^*| & \;\; \;\;  \text{if} ~  U''(R) = \emptyset\\
 
\max \{2|\textnormal{Id}_{\perp}(R)^*|, |\textnormal{Id}(R)^*| \} & \;\;  \;\; \text{if}  ~ U''(R) \neq \emptyset.
 \end{cases}
    \]
   Suppose that $S$ is any arbitrary independent set of the graph $\textnormal{Cl}_2(Y)_{SR}$. Then, we show that $|S| \le m$.  For each $u \in U(R)$, consider the set 
    \[ S_u = \{ (e,u) \in S ~ | ~ \text{for some} ~ e \in  \textnormal{Id}(R) \}. \]
    Note that $S_u$ forms a partition of the set $S$. Now if $S = S_u$ for some $u \in U'(R)$, then $|S| \le |\textnormal{Id}(R)^*|$. It follows that $|S| \le m$. If $S = S_u \cup S_u^{-1}$ for some $u \in U''(R)$, then for all $(e,u) \in S_u$, observe that $e \in \textnormal{Id}_{\perp}(R)^*$. It follows that $|S_u| \le |\textnormal{Id}_{\perp}(R)^*|$. Similarly, $|S_u^{-1}| \le |\textnormal{Id}_{\perp}(R)^*|$. Consequently, $|S| \le 2  |\textnormal{Id}_{\perp}(R)^*|$ and so $|S| \le m$. Now, suppose that the above are not the cases for the set $S$. For each $u \in U(R)$, consider the set
    \[
E_u = \{ e \in  \textnormal{Id}(R) ~ | ~  (e,u) \in S_u \}.
    \]
   Now, define a set $C \subseteq U''(R)$ such that for each $u \in U''(R)$, one of the elements $u$ or $u^{-1}$ belongs to $C$. Let $D = \bigcup \{ E_u ~ | ~ u \in U'(R) \cup C \}$. Observe that all the elements of the set $D$ are orthogonal idempotents of the ring $R$. Thus, $|D| \le |\textnormal{Id}_{\perp}(R)^*|$. Now consider the set $D' = \{ (f,v) \in S ~ | ~ v \in U''(R) \setminus C\}$. Let $(h', w') \in S$. If there does not exists $(h'',w'') \in S$ such that $w'w'' =1$, then take $h' \in D$ i.e. $w' \in C$. Suppose that $(f_1, v_1) \in D'$. Then there exists $(e,u) \in S$ such that $e \in D$ and $u v_1 = 1$.
   Also, if $(f_2, v_2) \in D'$, then $v_1 v_2 \neq 1$ and so $f_1 f_2 = 0$. It implies that $f \in \textnormal{Id}_{\perp}(R)^*$ for all $(f, v) \in D'$. Consequently, $|D'| \le |\textnormal{Id}_{\perp}(R)^*|$. Let $(e, u_1)$ and $(e, u_2) \in S$, where $u_1, u_2 \in U'(R) \cup C$. Then note that $e \in D$ and $u_1 =u_2$, i.e.,  $(e, u_1) = (e, u_2)$. Consequently, $|S| = |D| + |D'|$. It follows that $|S| \le m$. This, $I_u$ is the maximal independent set of the maximum size of the graph $\textnormal{Cl}_2(Y)_{SR}$ and so $\beta(\textnormal{Cl}_2(Y)_{SR}) = \max \{2|\textnormal{Id}_{\perp}(R)^*|, |\textnormal{Id}(R)^*| \}$. Note that $|X'| \le \beta(\textnormal{Cl}_2(Y)_{SR})$. Now, based on the cardinality of the set $U''(R)$, we calculate $|X'|$ and $|X_{(1)}|$. For this, we have the following sub-cases.

   \noindent\textbf{Sub-case (i).} $|U''(R)| =2$. If $\max \{2|\textnormal{Id}_{\perp}(R)^*|, |\textnormal{Id}(R)^*| \} = |\textnormal{Id}(R)^*|$, then take $X' = I_u$, where $u \in U'(R)$. Consider the set $\{(1,v), (1, v^{-1}) ~ | ~ v \in U''(R) \}$. Note that $I_u \cup \{(1,v), (1, v^{-1}) \} $ is an independent set of the graph $H'$. Now choose $X_{(1)} = \{(1,v), (1, v^{-1}) \}$ and so $|X_{(1)}| = 2 $. Consequently, $\beta(\textnormal{Cl}_2(R)_{SR}) = |I_u| + |X_{(1)}| = |\textnormal{Id}(R)^*| +2$.
   
   Next, let $\max \{2|\textnormal{Id}_{\perp}(R)^*|, |\textnormal{Id}(R)^*| \} = 2|\textnormal{Id}_{\perp}(R)^*|$. Suppose that $(e,w), (f, w^{-1}) \in X'$, where $w \in U''(R)$. Then both the vertices $(1,w)$ and $(1,w^{-1})$ does not belong to $X_{(1)}$. Consequently, $|X_{(1)}|  \le 1$. Choose $X' =I_u$, where $u \in U''(R)$ and take $X_{(1)} = \{(1,v)\}$, where $v \in U'(R)$. Then, $I_u \cup \{(1,v) \}$ is an independent set of maximum size of $\textnormal{Cl}_2(R)_{SR}$ and so $\beta(\textnormal{Cl}_2(H')_{SR}) = |I_u| + |X_{(1)}| = 2|\textnormal{Id}_{\perp}(R)^*| +1$. Now suppose that there does not exist $e, f \in \textnormal{Id}(R)^*$ such that $(e,w), (f, w^{-1}) \in X'$. Then note that $|X'| \le |\textnormal{Id}(R)^*|$. Take $X' = I_u$ for some $u \in U'(R)$, then by the previous argument, we have $I_u \cup \{(1,w), (1, w^{-1}) \}$, where $w \in U''(R)$, is an independent set of maximum size of the graph $H'$ and so $|X_{(1)}| =2$. Since $\max \{2|\textnormal{Id}_{\perp}(R)^*|, |\textnormal{Id}(R)^*| \} = 2|\textnormal{Id}_{\perp}(R)^*|$, we have $|\textnormal{Id}(R)^*| \le 2|\textnormal{Id}_{\perp}(R)^*| -1$. Consequently, $|X|$ to be maximum, we must have $|\textnormal{Id}(R)^*| = 2|\textnormal{Id}_{\perp}(R)^*| -1$ and so $|X'| = 2|\textnormal{Id}_{\perp}(R)^*| -1$. Consequently, $\beta(\textnormal{Cl}_2(R)_{SR}) = |I_u| + |X_{(1)}| = 2|\textnormal{Id}_{\perp}(R)^*| +1$.

   \noindent\textbf{Sub-case (ii).} $|U''(R)| > 2$. Let $I_u$ be the maximal independent set of the graph $\beta(\textnormal{Cl}_2(Y)_{SR})$ for some $u \in U(R)$ i.e. $X' = I_u$. Since $|U''(R)| > 2$, we can always choose $v \in U''(R)$ such that $v \notin \{u, u^{-1} \}$. Consider the set $ \{ (1,v), (1,v^{-1}) \} $. Note that $I_u \cup \{(1,v), (1, v^{-1}) \} $ is a independent set of the graph $H'$. Now choose $X_{(1)} = \{(1,v), (1, v^{-1}) \}$ and so $|X_{(1)}| = 2 $. Consequently, $\beta(\textnormal{Cl}_2(R)_{SR}) = |I_u| + |X_{(1)}| = \max \{2|\textnormal{Id}_{\perp}(R)^*|, |\textnormal{Id}(R)^*| \} +2$.
\end{proof}

\begin{theorem}\label{relationclrcl2r}
Let $R$ be a commutative ring with some non-trivial idempotent elements. Then the following holds.

    \begin{enumerate}[\rm(i)]
        \item If $|U(R)| = 1$, then $\beta(\textnormal{Cl}(R)_{SR}) = \beta(\textnormal{Cl}_2(R)_{SR}) = 1$.
        \item If $|U(R)| \ge 2$ and $U''(R) = \emptyset$, then $\beta(\textnormal{Cl}(R)_{SR}) = \beta(\textnormal{Cl}_2(R)_{SR}) + 1$.
        \item If $|U''(R)| =2 $, then $\beta(\textnormal{Cl}(R)_{SR}) = \beta(\textnormal{Cl}_2(R)_{SR}) $.
         \item If $|U''(R)| > 2$ and $\max \{2|\textnormal{Id}_{\perp}(R)^*|, |\textnormal{Id}(R)^*| \} = |\textnormal{Id}(R)^*|$, then $\beta(\textnormal{Cl}(R)_{SR}) = \beta(\textnormal{Cl}_2(R)_{SR}) $.
        \item If $|U''(R)| > 2$ and $\max \{2|\textnormal{Id}_{\perp}(R)^*|, |\textnormal{Id}(R)^*| \} = 2|\textnormal{Id}_{\perp}(R)^*|$, then $\beta(\textnormal{Cl}(R)_{SR}) =  \beta(\textnormal{Cl}_2(R)_{SR}) -1$.
    \end{enumerate}
\end{theorem}

\begin{proof}
      If $|U(R)| =1$, then $\textnormal{Cl}(R)$ and $\textnormal{Cl}_2(R)$ are complete graphs. Consequently, $\beta(\textnormal{Cl}(R)_{SR}) = \beta(\textnormal{Cl}_2(R)_{SR}) = 1$. Next, we assume that $|U(R)| \ge 2$. Observe that $K \cong \textnormal{Cl}_2(Y)_{SR}$, where $Y = \{(e,u) \in V(\textnormal{Cl}_2(R)_{SR}) ~ | ~ e \neq 1 \}$ and so $\beta(\textnormal{Cl}_2(Y)_{SR}) = \beta(K)$. By the definition of the graphs $G$ and $K$, we have that $K$ is an induced subgraph of $G$. Thus, every independent set of $K$ is the independent set of $G$. Consequently, 
  \begin{equation}\label{eqa}
   \beta(K) \le \beta(G).
  \end{equation} 
  Now, we prove the results in the following cases.
  
  \noindent\textbf{Case-1:} $U''(R) = \emptyset$. Let $S$ be an independent set of maximum size of $G$. Consider the set $S' = \{ (1,u) \in V(G) ~ | ~ (1,u) \in S \}$. Then $S'$ is the subset of $S$, and note that $S \setminus S'$ is the independent set of the graph $K$. Let $(1,v_1)$ and $(1,v_2)$ are two distinct elements of $S'$. Then, by the adjacency of the graph $G$, we have $v_1 v_2 =1$. Since $U''(R) = \emptyset$, we conclude that $v_1=v_2$ and so $|S'| \le 1$. Consequently, $\beta(G) \le \beta(K) +1$. By the proof of the Theorem \ref{independentsetcl2r}, we know that $ I_u = \{(e,u) ~ | ~ e \in \textnormal{Id}(R)^* \}$, where $u \in U'(R)$, is an independent set of the maximum size of the graph $K$. Then, we can choose $(1,u) \in S'$ such that $S = I_u \bigcup \{ (1,u) \}$. Thus,
\begin{equation}\label{eq1}
\beta(G) = \beta(K) +1.
\end{equation}
By Theorem \ref{graphgclr}, we have $\beta(\textnormal{Cl}(R)_{SR}) = \beta(G) +\beta(K_{|U(R)|}) = \beta(G) +1$. Also by Theorem \ref{graphgcl2r}, we obtain $\beta(\textnormal{Cl}_2(R)_{SR}) = \beta(K) +\beta(K_{|U(R)|}) = \beta(K) +1$. By using the equation (\ref{eq1}), we have
\[ \beta(G) +1 = \beta(K) +1 +1. \]
It follows that
  \[ \beta(\textnormal{Cl}(R)_{SR}) = \beta(\textnormal{Cl}_2(R)_{SR}) + 1.\]
  
\noindent\textbf{Case-2:} $U''(R) \neq \emptyset$. By the Theorems \ref{graphgcl2r} and \ref{graphgclr}, we have that $\beta(\textnormal{Cl}_2(R)_{SR}) = \beta(H')$ and $\beta(\textnormal{Cl}(R)_{SR}) = \beta(G) +1$. Now, consider the sets $S$ and $S'$ as in Case-1. Now suppose that $(1,u)$ and $(1,v)$ are two distinct elements of the set $S'$, then $u v =1$. It implies that $|S'| \le 2$. Note that $S \setminus S'$ is the independent set of the graph $K$. Suppose that $|S'| =2$. Then $S' = \{ (1,v), (1,v^{-1}) \}$  for some $v \in U''(R)$. Now assume that $(e,w) \in S \setminus S'$, then $w v =1$ and $w v^{-1} = 1$, which is not possible. Consequently, $S \setminus S' = \emptyset$, and so $S = S'$. But we can always choose $e \in \textnormal{Id}(R)^*$ and $u' \in U'(R)$ such that $\{ (e,u'), (1-e, u'), (1, u') \} \subseteq S$ and so $|S| \ge 3$. Thus, our assumption $|S'| =2$ is not possible. Hence, $|S'| \le 1$. It follows that 
     \begin{equation}\label{eq2}
   \beta(G) \le \beta(K) +1.
  \end{equation}
  By the equations (\ref{eqa}) and (\ref{eq2}), we have

      \begin{equation}\label{eq3}
   \beta(K) \le \beta(G) \le \beta(K) +1.
  \end{equation}

Now, assume that $\max \{2|\textnormal{Id}_{\perp}(R)^*|, |\textnormal{Id}(R)^*| \} = |\textnormal{Id}(R)^*|$. Note that as the Case-1, for any $u \in U'(R)$, we can choose $S = I_u \bigcup \{ (1,u) \}$, where $ I_u = \{(e,u) ~ | ~  e \in \textnormal{Id}(R)^* \}$ is an independent set of maximum size of the graph $K$. Consequently, $\beta(G) = \beta(K) +1$. Thus, we have
\begin{equation}\label{eqA}
 \beta(\textnormal{Cl}(R)_{SR}) =\beta(G) +1 = \beta(K) +2.
    \end{equation}
 
  Next, assume that $\max \{2|\textnormal{Id}_{\perp}(R)^*|, |\textnormal{Id}(R)^*| \} = 2|\textnormal{Id}_{\perp}(R)^*|$. Then, by the proof of the Theorem \ref{independentsetcl2r}, there exists $e \in \textnormal{Id}(R)^*$ and $w \in U''(R)$ such that both $(e,w)$ and $(e, w^{-1})$ belongs to the set $S \setminus S'$ (independent set of the graph $K$). Now if $(1, w') \in S'$, then $w' w=1$ and $w' w^{-1} =1$, which is not possible. Thus, the set $S'$ is an empty set. Consequently, $S$ is an independent set of maximum size of the graph $G$ and $K$. Thus, $\beta(G) = \beta(K)$. Consequently, we have 
\begin{equation}\label{eqB}
   \beta(\textnormal{Cl}(R)_{SR}) =\beta(G) +1 = \beta(K) +1.
   \end{equation}

By the proof of the Theorem \ref{independentsetcl2r}, we have

\[ \beta(\textnormal{Cl}_2(R)_{SR}) =
\begin{cases}
  \beta(\textnormal{Cl}_2(Y)_{SR}) +2  &  ~~  \textnormal{if } ~~ |U''(R)| =2 ~ \text{and} ~ \max \{2|\textnormal{Id}_{\perp}(R)^*|, |\textnormal{Id}(R)^*| \} = |\textnormal{Id}(R)^*|, \\
 \beta(\textnormal{Cl}_2(Y)_{SR}) +1 & ~~ \textnormal{if } ~~ |U''(R)| =2 ~ \text{and} ~ \max \{2|\textnormal{Id}_{\perp}(R)^*|, |\textnormal{Id}(R)^*| \} = 2|\textnormal{Id}_{\perp}(R)^*|, \\
 \beta(\textnormal{Cl}_2(Y)_{SR}) +2 & ~~ \textnormal{if } ~~ |U''(R)| > 2.
\end{cases}
\]
Using the fact $\beta(\textnormal{Cl}_2(Y)_{SR}) = \beta(K)$, we have
\begin{equation}\label{eqC}
    \beta(\textnormal{Cl}_2(R)_{SR}) =
\begin{cases}
  \beta(K) +2  &  ~~  \textnormal{if } ~~ |U''(R)| =2 ~ \text{and} ~ \max \{2|\textnormal{Id}_{\perp}(R)^*|, |\textnormal{Id}(R)^*| \} = |\textnormal{Id}(R)^*|, \\
 \beta(K) +1 & ~~ \textnormal{if } ~~ |U''(R)| =2 ~ \text{and} ~ \max \{2|\textnormal{Id}_{\perp}(R)^*|, |\textnormal{Id}(R)^*| \} = 2|\textnormal{Id}_{\perp}(R)^*|, \\
 \beta(K) +2 & ~~ \textnormal{if } ~~ |U''(R)| > 2.
\end{cases}
\end{equation}
Combining the equations (\ref{eqA}), (\ref{eqB}) and (\ref{eqC}), we have the desired result.
\end{proof}

    In Theorem \ref{relationclrcl2r}, we derived the relation between the independence number of the graphs $\textnormal{Cl}(R)_{SR}$ and $\textnormal{Cl}_2(R)_{SR}$, when the ring $R$ has some non-trivial idempotents. Using Theorem \ref{independentsetcl2r} and Theorem \ref{relationclrcl2r}, one can easily calculate the independence number of the graph $\textnormal{Cl}(R)_{SR}$. Next, we calculate the independence number of the graph $\textnormal{Cl}(R)_{SR}$, when $R$ has no non-trivial idempotents.

    \begin{remark}
    Note that the independence number is the maximum cardinality of the independence set. Therefore, whenever $2|\textnormal{Id}_{\perp}(R)^*| = |\textnormal{Id}(R)^*|$, we have $\beta(\textnormal{Cl}_2(R)_{SR}) = |\textnormal{Id}(R)^*| +2$ and $\beta(\textnormal{Cl}(R)_{SR}) = \beta(\textnormal{Cl}_2(R){SR})$.
\end{remark}

    \begin{theorem}
   Let $R$ be a commutative ring with no non-trivial idempotent such that $\textnormal{Cl}(R)$ is not a complete graph. Then $\beta(\textnormal{Cl}(R)_{SR}) = 2$.
\end{theorem} 
\begin{proof}
    The proof is easy to observe by Theorem \ref{clrsrnoidempotent}.
\end{proof}

\section{Strong Metric Dimension of Clean graphs}\label{section4}

In this section, we investigate the strong metric dimension of the graphs $\textnormal{Cl}(R)$ and $\textnormal{Cl}_2(R)$. Further, we study the relation between the strong metric dimension of the graphs $\textnormal{Cl}(R)$ and $\textnormal{Cl}_2(R)$. We begin with the following theorem.

\begin{theorem}
    The strong metric dimension of $\textnormal{Cl}(R)$ $(\text{or} ~ \textnormal{Cl}_2(R))$ is finite if and only if the graph $\textnormal{Cl}(R)$ $(\text{or} ~ \textnormal{Cl}_2(R))$ is finite.  
\end{theorem}

\begin{proof}
    It is clear that if $\textnormal{Cl}(R)$ is finite, then $\textnormal{sdim(Cl}(R))$ is finite. To prove the other side of the result, let $W = \{ x_1, x_2, \ldots , x_n \}$ be a finite strongly resolving set of  $\textnormal{Cl}(R)$. It implies that, for all $x,y \in V( \textnormal{Cl}(R))$, there exists $z \in W$ such that $z$ strongly resolves $x$, $y$ and so $d(x,z) \neq d(y,z)$. Let $D(x|W)$ be the sequence of distances of $x$ with the elements of $W$, i.e., $D(x|W) = (d(x,x_1), d(x,x_2), \ldots, d(x,x_n))$. For distinct $x,y \in V( \textnormal{Cl}(R)) \setminus W$ note that $D(x|W) \neq D(y|W)$. Since $\textnormal{diam}( \textnormal{Cl}(R)) \le 2$, there are at most $2^n$ choices for each $D(x|W)$. Hence, $|V( \textnormal{Cl}(R))| \le 2^n +n$. Thus, the result holds for $\textnormal{Cl}(R)$. Since $\textnormal{diam}( \textnormal{Cl}_2(R)) \le 3$, similarly we get the result for the graph $\textnormal{Cl}_2(R)$. 
\end{proof}

\begin{remark}
The graphs $\textnormal{Cl}(R)$ and $\textnormal{Cl}_2(R)$ are finite if and only if both the sets $\textnormal{Id}(R)$ and $U(R)$ are finite.   
\end{remark}
 
In what follows, we calculate the strong metric dimension of the clean graphs. For this, we assume $\textnormal{Cl}(R)$ and $\textnormal{Cl}_2(R)$ to be the finite graphs. Note that $|V(\textnormal{Cl}_2(R))| = (|\textnormal{Id}(R)|-1)|U(R)|$. By Theorems \ref{vertexcover}, \ref{Gallaithoerem} and \ref{independentsetcl2r}, one can calculate the strong metric dimension of the graph $\textnormal{Cl}_2(R)$. Using the fact $|V(\textnormal{Cl}(R))| = |V(\textnormal{Cl}_2(R))| + |U(R)|$ and Theorem \ref{relationclrcl2r}, we obtain the following theorem.

\begin{theorem}\label{sdimrelation}
Let $R$ be a commutative ring with some non-trivial idempotents such that $\textnormal{Cl}(R)$ is finite. Then the following holds.

    \begin{enumerate}[\rm(i)]
        \item If $|U(R)| = 1$, then $\textnormal{sdim(Cl}(R)) = \textnormal{sdim(Cl}_2(R)) + 1$.
        \item If $|U(R)| \ge 2$ and $U''(R) = \emptyset$, then $\textnormal{sdim(Cl}(R)) = \textnormal{sdim(Cl}_2(R)) + |U(R)| - 1$.
        \item If $|U''(R)| =2 $, then $\textnormal{sdim(Cl}(R)) = \textnormal{sdim(Cl}_2(R)) + |U(R)|$.
         \item If $|U''(R)| > 2$ and $\max \{2|\textnormal{Id}_{\perp}(R)^*|, |\textnormal{Id}(R)^*| \} = |\textnormal{Id}(R)^*|$, then $\textnormal{sdim(Cl}(R)) = \textnormal{sdim(Cl}_2(R)) + |U(R)|$.
        \item If $|U''(R)| > 2$ and $\max \{2|\textnormal{Id}_{\perp}(R)^*|, |\textnormal{Id}(R)^*| \} = 2|\textnormal{Id}_{\perp}(R)^*|$, then $\textnormal{sdim(Cl}(R)) = \textnormal{sdim(Cl}_2(R)) + |U(R)| + 1$.
        
    \end{enumerate}
\end{theorem}

In the following theorem, we calculate the strong metric dimension of the graph $\textnormal{Cl}(R)$ when $R$ contains only trivial idempotents.

\begin{theorem}\label{sdimnoidempotemts}
 Let $R$ be a ring with no non-trivial idempotent. Then the following holds.

 \begin{enumerate}[\rm(i)]
     \item If $|U(R)| = 1$, then $\textnormal{sdim(Cl}(R)) = 1$.
     \item If $|U(R)| \ge 2$, then $\textnormal{sdim(Cl}(R)) = 2|U(R)|-2$.
 \end{enumerate}
\end{theorem}

\begin{proof}
   First, we calculate the independence number of the graph $\textnormal{Cl}(R)_{SR}$. If $|U(R)| =1$, then by Theorem \ref{relationclrcl2r}, $\beta(\textnormal{Cl}(R)_{SR})=1$. Assume that $|U(R)| \ge 2$. By Theorem \ref{clrsrnoidempotent},
   \begin{equation}\label{eq6}
        \beta(\textnormal{Cl}(R)_{SR}) = \beta(2 K_{|U(R)|}) = 2.
   \end{equation}
     Note that $V(\textnormal{Cl}(R)_{SR}) = \{ (0, u) ~ | ~ u \in U(R) \} \cup \{ (1, u) ~ | ~ u \in U(R \}$. Consequently, $|V(\textnormal{Cl}(R)_{SR})| = 2 |U(R)|$. By Theorem \ref{vertexcover} and Theorem \ref{Gallaithoerem}, the result holds.
 \end{proof}

\subsection{Strong Metric Dimension of the Clean graphs of Artinian rings}\label{section5}

In this subsection, we study the strong metric dimension of the clean graphs of commutative Artinian rings. By the structural theorem \cite{atiyah1969introduction}, an Artinian non-local commutative ring $R$ is uniquely (up to isomorphism) a finite direct product of local rings $R_i$ that is $R \cong R_1 \times R_2 \times \cdots \times R_n$. Also, observe that $|\textnormal{Max}(R)| =n$, when $R \cong R_1 \times R_2 \times \cdots \times R_n$ $(n \ge 1)$. The following theorem determines the Artinian rings for which the clean graphs are finite.

\begin{theorem}
    Let $R$ be a commutative Artinian ring. Then the graphs $\textnormal{Cl}(R)$ and $\textnormal{Cl}_2(R)$ are finite if and only if the ring $R$ is finite. 
\end{theorem}

\begin{proof}
     Let $R$ be a commutative Artinian ring. If $R$ is finite, then it is obvious that the graphs $\textnormal{Cl}(R)$ and $\textnormal{Cl}_2(R)$ are finite graphs. Now suppose that $\textnormal{Cl}(R)$ (or $\textnormal{Cl}_2(R)$) is a finite graph. It implies that the sets Id$(R)$ and $U(R)$ are finite. We know that an Artinian ring can be written as a finite product of local Artinian rings. It implies that $R=R_1 \times R_2 \times \cdots \times R_n$, where each $R_i$ is a local Artinian ring with maximal ideal $\mathcal{M}_i$. Assume that $n=1$ i.e., $R(=R_1)$ is a local ring. We know that for all $x \in M_1$, we have $1+x$ is a unit of $R_1$. Since $U(R_1)$ is finite, we have $\mathcal{M}_1$ is finite. It implies that $R_1$ is finite. Similarly, if $n \geq 2$, we know that $|U(R)|=|U(R_1)| \times|U(R_2)| \times \cdots \times |U(R_n)|$. Since $U(R)$ is finite, it gives us $U\left(R_i\right)$ is finite for each $i$. By the previous argument, we conclude that each $R_i$ is finite. Hence, $R$ is finite.
\end{proof}

By previous theorems, we conclude that for the strong metric dimension of clean graphs of an Artinian ring $R$ to be finite, we must have $R$ as a finite ring. Also, every finite ring is an Artinian ring. So, we assume $R$ to be a finite ring for the remaining part of the paper. Note that the local ring has only trivial idempotents. When $R$ is an Artinian local ring i.e., $|Max(R)|=1$, then the following theorem is the direct consequence of the Theorem \ref{sdimnoidempotemts}.

\begin{theorem}\label{sdimlocal}
 Let $R$ be a finite commutative local ring. Then the following holds.

 \begin{enumerate}[\rm(i)]
     \item If $|U(R)| = 1$, then $\textnormal{sdim(Cl}(R)) = 1$.
     \item If $|U(R)| \ge 2$, then $\textnormal{sdim(Cl}(R)) = 2|U(R)|-2$.
 \end{enumerate}
\end{theorem}

 We have the following corollary as the consequence of the Theorem \ref{sdimlocal}.

 \begin{corollary}
    Let $\mathbb{F}_q$ be a finite field of order $q$. Then the following holds.
     \begin{enumerate}[\rm(i)]
     \item If $q = 2$, then $\textnormal{sdim(Cl}(R)) = 1$.
     \item If $q \ge 3$, then $\textnormal{sdim(Cl}(R)) = 2q-4$.
 \end{enumerate}
    
\end{corollary}

Now we investigate the strong metric dimension of a commutative Artinian ring $R$ when $|\textnormal{Max}(R)| \ge 2$. Thus, $R \cong R_1 \times R_2 \times \cdots \times R_n$ $(n \ge 2)$, where each $R_i$ is a local ring. Consequently, $|\textnormal{Id}(R)| = 2^n$. Therefore, $|V(\textnormal{Cl}(R))| = 2^n\prod_{i=1}^{n} |U(R_i)|$ and $|V(\textnormal{Cl}_2(R))| =  |V(\textnormal{Cl}(R))| - |U(R)| = $ $(2^n-1)\prod_{i=1}^{n} |U(R_i)|$. We begin with the following lemma.

\begin{lemma}\label{twoinvertible}
    Let $R \cong R_1 \times R_2 \times \cdots \times R_n$ $(n \ge2)$ be a non-local commutative Artinian ring with unity. Then $|U''(R)| =2$ if and only if $R_i \cong \mathbb{Z}_2$ for each $i \in \{2,3, \ldots,n \}$ and $R_1$ is isomorphic to one of the following four rings:
    \[ \mathbb{F}_4, ~ \mathbb{Z}_5, ~ \dfrac{\mathbb{Z}_2[x]}{\langle x^3 \rangle}, ~ \dfrac{\mathbb{Z}_4[x]}{\langle 2x, x^2-2\rangle}. \]
\end{lemma}

\begin{proof}
    First, let $|U''(R)| =2$. Then by Proposition \ref{twononselfinvertible}, we have either $|U(R)| =3$ or $|U(R)| =4$. Note that $|U(R)| = \prod_{i=1}^{n} |U(R_i)|$. First, assume that $|U(R)| =3$. Then there exist $i \in \{1,2, \ldots,n \}$ such that $|U(R_i)| =3$ and $|U(R_j)| =1$ for each $j \neq i$ and so $R_j \cong \mathbb{Z}_2$ for all $j \neq i$. Without loss of generality, assume that $|U(R_1)| =3$. If $|R_1| \ge 5$, then $|U(R_1)| \ge 4$, which is not possible. It implies that $|R_1|=4$. If $R_1$ is not a field, then $|U(R_1)|=2$ and so $U''(R) = \emptyset$. Consequently, $R_1$ is a field of order $4$.

    Next, assume that $|U(R)| =4$. Note that if $|U(R_i)|=2$ for some $i$, then $U''(R_i) = \emptyset$. Also, if $U''(R_i) = \emptyset$ for each $i$, then $U''(R) = \emptyset$. Therefore, there exist $i \in \{1,2, \ldots,n \}$ such that $|U(R_i)| =4$ and $|U(R_j)| =1$ for each $j \neq i$ and so $R_j \cong \mathbb{Z}_2$ for all $j \neq i$. Without loss of generality, assume that $|U(R_1)| =4$. Note that if $|R_1| \ge 9$, then $|U(R_1)| > 4$, which is not possible. It implies that either $R_1$ is a field of order $5$ or $R_1$ is a local ring of order $8$ such that $|U''(R_1)| =2$. By Table \ref{localrings}, we obtain that $R_1$ is isomorphic to one of the rings: $\mathbb{Z}_5$, $\dfrac{\mathbb{Z}_2[x]}{\langle x^3 \rangle}$, $\dfrac{\mathbb{Z}_4[x]}{\langle 2x, x^2-2\rangle}$.

    Converse is straightforward.
\end{proof}


\begin{lemma}\label{sdimofcl2r}
  Let $R \cong R_1 \times R_2 \times \cdots \times R_n$ $(n \ge2)$, where each $R_i$ is a local ring, be a non-local commutative Artinian ring and let $U''(R) = \emptyset$. Then 
  \begin{enumerate}[\rm(i)]
        \item If $|U(R)|=1$, then $\textnormal{sdim(Cl}_2(R)) = 2^n- 2$.
        \item If $|U(R)| \ge 2$, then $\textnormal{sdim(Cl}_2(R)) = (2^n-1)|U(R)| - 2^n +1$.
    \end{enumerate}
\end{lemma}

\begin{proof}
 Let $R \cong R_1 \times R_2 \times \cdots \times R_n$ $(n \ge2)$ be a non-local commutative Artinian ring, where each $R_i$ is a local ring. We know that the local ring has only $\{0,1 \}$ as the set of idempotent elements. Thus, each idempotent element of $R$ is of the form $(a_1, a_2, \ldots, a_n)$ such that $a_i \in \{0,1\}$ for each $1 \le i \le n$. Consequently, $|\textnormal{Id}(R)^*| = 2^n -2$. If $|U(R)| =1$, then by Theorem \ref{relationclrcl2r}, we have $\beta(\textnormal{Cl}_2(R)_{SR}) = 1$. Assume that $U''(R) = \emptyset$. By Theorem \ref{independentsetcl2r}, we have $\beta(\textnormal{Cl}_2(R)_{SR}) = |\textnormal{Id}(R)^*| +1 = 2^n -1$. Thus, by Theorem \ref{vertexcover} and Theorem \ref{Gallaithoerem}, the result holds.  
\end{proof}

 The following corollary calculates the strong metric dimension of the graph $\textnormal{Cl}_2(R)$ when $R$ is an Artinian reduced ring i.e., $R$ is a direct product of finite fields. Note that when $R \cong F_1 \times F_2 \times \cdots \times F_n$ $(n \ge2)$, then $|\textnormal{Id}(R)| = 2^n$ and $|U(R)| = \prod_{i=1}^{n} (|F_i|-1)$. Thus,  $|V(\textnormal{Cl}(R))| = 2^n\prod_{i=1}^{n} (|F_i|-1)$ and $|V(\textnormal{Cl}_2(R))| =  |V(\textnormal{Cl}(R))| - |U(R)| = (2^n-1)\prod_{i=1}^{n} (|F_i|-1)$.
\begin{corollary}
    Let $R \cong F_1 \times F_2 \times \cdots \times F_n$ $(n \ge2)$, where each $F_i$ is a field. Then the following holds
    \begin{enumerate}[\rm(i)]
        \item If $F_i \cong \mathbb{Z}_2$ for all $ 1 \le i \le n$, then $\textnormal{sdim(Cl}_2(R)) = 2^n-2$.
        \item If $F_i$ is either $\mathbb{Z}_2$ or $\mathbb{Z}_3$ such that $R \ncong \mathbb{Z}_2 \times \mathbb{Z}_2 \times \cdots \times \mathbb{Z}_2 $, then $\textnormal{sdim(Cl}_2(R)) = (2^n-1)\prod_{i=1}^{n} (|F_i|-1) - 2^n +1$.
    \end{enumerate}
\end{corollary}

\begin{proof}
    Let $R \cong  F_1 \times F_2 \times \cdots \times F_n$ $(n \ge2)$, where each $F_i$ is a field. Note that $U''(R) = \emptyset$ if and only if $U''(F_i) = \emptyset$. Thus, by the proposition \ref{selfinvertibleunits}, $U''(R) = \emptyset$ if and only if each $F_i$ is isomorphic to either $\mathbb{Z}_2$ or $\mathbb{Z}_3$. Hence, by the Theorem \ref{betaofcl2r}, the result holds. 
\end{proof}

\begin{theorem}\label{betaofcl2r}
    Let $R \cong R_1 \times R_2 \times \cdots \times R_n$ $(n \ge2)$, where each $R_i$ is a local ring, be a non-local commutative Artinian ring and let $U''(R) \neq \emptyset$. Then
    \begin{enumerate}[\rm(i)]
        \item If $n \ge 3$, then $\textnormal{sdim(Cl}_2(R)) = (2^n-1)|U(R)| - 2^n$. 
        \item If $n=2$ and $R$ is isomorphic to one of the rings: $\mathbb{F}_4 \times \mathbb{Z}_2$, $\mathbb{Z}_5 \times \mathbb{Z}_2$, $\dfrac{\mathbb{Z}_2[x]}{\langle x^3 \rangle} \times \mathbb{Z}_2$, $\dfrac{\mathbb{Z}_4[x]}{\langle 2x, x^2-2\rangle} \times \mathbb{Z}_2$, then $\textnormal{sdim(Cl}_2(R)) = 3|U(R)| - 5$. 
        \item If $n=2$ and $R$ is not isomorphic to the rings given in {\rm(ii)}, then $\textnormal{sdim(Cl}_2(R)) = 3|U(R)| - 6$.

    \end{enumerate}
\end{theorem}



\begin{proof}
Let $R$ be a commutative Artinian ring such that $|\textnormal{Max}(R)| \ge 2$. Then  $R \cong R_1 \times R_2 \times \cdots \times R_n$ $(n \ge2)$, where each $R_i$ is a local ring. By proof of the Lemma \ref{sdimofcl2r}, $|\textnormal{Id}(R)^*| = 2^n -2$. Also, by the Pigeonhole principle, the set of nontrivial orthogonal idempotents of maximum size is 
\[ 
\textnormal{Id}_{\perp}(R)^* = \{ e_i = (a_1, a_2, \ldots, a_{i-1}, a_i, a_{i+1}, \ldots, a_n) ~|~ a_i = 1 ~ \text{and} ~ a_j = 0 ~ \text{for} ~ j \neq i \}. 
\]
Consequently, $|\textnormal{Id}_{\perp}(R)^*| =n$. Let $U''(R) \neq \emptyset$. Notice that $\max\{2|\textnormal{Id}_{\perp}(R)^*|, |\textnormal{Id}(R)^*| \} = \max\{2n, 2^n -2 \} = 2n$ if and only if $n=2,3$. Also, note that for $n=3$, $2|\textnormal{Id}_{\perp}(R)^*|= |\textnormal{Id}(R)^*|$. Thus, for $n=2$, Theorem \ref{independentsetcl2r} implies that $\beta(\textnormal{Cl}_2(R)_{SR}) =5$ when $|U''(R)| = 2$ and $\beta(\textnormal{Cl}_2(R)_{SR}) = 6$ when $|U''(R)|  >2$. Moreover, for $n\ge 3$, we have $\beta(\textnormal{Cl}_2(R)_{SR}) = 2^n$. By Lemma \ref{twoinvertible} and Theorems \ref{vertexcover}, \ref{Gallaithoerem}, the result holds.
\end{proof}

\begin{corollary}\label{reducedcl2r}
    Let $R \cong F_1 \times F_2 \times \cdots \times F_n$ $(n \ge2)$, where each $F_i$ is a field and $|F_i| \ge 4$ for some $i$. Then the following holds
    \begin{enumerate}[\rm(i)]
        \item If $R$ is isomorphic to either $ \mathbb{Z}_2 \times \mathbb{F}_4$ or $\mathbb{Z}_2 \times \mathbb{Z}_5$, then $\textnormal{sdim(Cl}_2(R)) = 3\prod_{i=1}^{n} (|F_i|-1) -  5$.
         \item If $n=2$ and $R$ is not isomorphic the rings given in {\rm(i)}, then $\textnormal{sdim(Cl}_2(R)) = (2^n-1)\prod_{i=1}^{n} (|F_i|-1) - 6$.
        \item If $n \ge 3$, then $\textnormal{sdim(Cl}_2(R)) = (2^n-1)\prod_{i=1}^{n} (|F_i|-1) - 2^n$.
    \end{enumerate}
\end{corollary}


The following theorem is the consequence of the Theorem \ref{sdimrelation}, Lemma \ref{sdimofcl2r} and Theorem \ref{betaofcl2r}.


\begin{theorem}
    Let $R \cong R_1 \times R_2 \times \cdots \times R_n$ $(n \ge2)$, where each $R_i$ is a local ring, be a non-local commutative Artinian ring. Then the following holds
    \begin{enumerate}[\rm(i)]
        \item If $|U(R)|=1$, then $\textnormal{sdim(Cl}(R)) = 2^n - 1$.
        \item If $|U(R)| \ge 2$ and $n=2$, then $\textnormal{sdim(Cl}(R)) = 4|U(R)| - 5$.
        \item If $|U(R)| \ge 2$ and $n \ge 3$, then $\textnormal{sdim(Cl}(R)) = 2^n|U(R)| - 2^n$.
        
    \end{enumerate}
\end{theorem}

        


\begin{corollary}\label{reducedclr}
    Let $R \cong F_1 \times F_2 \times \cdots \times F_n$ $(n \ge2)$, where each $F_i$ is a field. Then the following holds
    \begin{enumerate}[\rm(i)]
        \item If $F_i = \mathbb{Z}_2$ for all $ 1 \le i \le n$, then $\textnormal{sdim(Cl}(R)) = 2^n - 1$.
        \item If $|F_i| \ge 3$ for some $i$. Then 
        \begin{enumerate}[\rm(a)]
            \item If $n =2$, then $\textnormal{sdim(Cl}(R)) = 4\prod_{i=1}^{n} (|F_i|-1) - 5$.
            \item If $n \ge 3$, then $\textnormal{sdim(Cl}(R)) = 2^n\prod_{i=1}^{n} (|F_i|-1) - 2^n$.
            \end{enumerate}
    \end{enumerate}
\end{corollary}

\textbf{Acknowledgement:} The first author gratefully acknowledges Birla Institute of Technology and Science (BITS) Pilani, India, for providing financial support.


\end{document}